\documentclass[reqno]{amsart}
\usepackage{epsfig}
\usepackage{amsmath}
\usepackage{amsfonts}
\usepackage{mathrsfs}
\begin{document}

\newtheorem{thm}{Theorem}[section]
\newtheorem{lem}[thm]{Lemma}
\newtheorem{cor}[thm]{Corollary}
\newtheorem{conj}[thm]{Conjecture}
\newtheorem{qn}[thm]{Question}
\newtheorem{pro}[thm]{Proposition}
\renewcommand{\theequation}{\arabic{section}.\arabic{equation}}
\allowdisplaybreaks

\def\square{\hfill${\vcenter{\vbox{\hrule height.4pt \hbox{\vrule width.4pt
height7pt \kern7pt \vrule width.4pt} \hrule height.4pt}}}$}

\title{Sharp Bounds for Generalized Elliptic Integrals of the First Kind}
\author{Miao-Kun Wang}

\address{Miao-Kun Wang, Department of Mathematics, Huzhou University, Huzhou 313000, China}

\email{wmk000@126.com}

\author{Yu-Ming Chu}

\address{Yu-Ming Chu (Corresponding author), Department of Mathematics, Huzhou University, Huzhou 313000, China}

\email{chuyuming@hutc.zj.cn}

\author{Song-Liang Qiu}

\address{Department of Mathematics, Zhejiang Sci-Tech University, Hangzhou, Zhejiang, 310018, China}

\email{sl\_qiu@zstu.edu.cn}

\thanks{This work was supported by Natural Science Foundation of China (Grant Nos. 61374086, 11171307) and Natural Science Foundation
of the Zhejiang Province (Grant No. LY13A010004).}



\begin{abstract} In this paper, we prove that the double inequality
\begin{equation*}
1+\alpha r'^2<\frac{\mathcal{K}_{a}(r)}{\sin(\pi
a)\log(e^{R(a)/2}/r')}<1+\beta r'^2
\end{equation*}
holds for all $a\in (0, 1/2]$ and $r\in (0, 1)$ if and only if
$\alpha\leq \pi/[R(a)\sin(\pi a)]-1$ and $\beta\geq a(1-a)$, where
$r'=\sqrt{1-r^2}$, $\mathcal{K}_{a}(r)$ is the generalized elliptic
integral of the first kind and $R(x)$ is the Ramanujan constant
function. Besides, as the key tool, the series expression for the
Ramanujan constant function $R(x)$ is given.
\end{abstract}

\maketitle
\section{Introduction}
For $r\in(0,1)$, Legendre's complete elliptic integrals [1] of the first
kind and the second kind are given by
\begin{equation*}
\mathcal{K}=\mathcal{K}(r)=\int_{0}^{\pi/2}\frac{dt}{\sqrt{1-r^2\sin^{2}(t)}},
\end{equation*}
\begin{equation*}
\mathcal{E}=\mathcal{E}(r)=\int_{0}^{\pi/2}\sqrt{1-r^2\sin^{2}(t)}dt,
\end{equation*}
respectively.  They are the particular cases of Gaussian
hypergeometric function
\begin{equation*}
F(a,b;c;x)=\sum_{n=0}^{\infty}\frac{(a)_{n}(b)_{n}}{(c)_{n}}\frac{x^n}{n!}\quad
(-1<x<1),
\end{equation*}
where $(a)_{n}=\Gamma(a+n)/\Gamma(a)$ and
$\Gamma(x)=\int_{0}^{\infty}t^{x-1}e^{-t}dt$ $(x>0)$ is the gamma
function. Indeed, we have
\begin{equation*}
\mathcal{K}(r)=\frac{\pi}{2}F\left(\frac{1}{2},\frac{1}{2};1;r^2\right),~~\mathcal{E}(r)=\frac{\pi}{2}F\left(-\frac{1}{2},\frac{1}{2};1;r^2\right).
\end{equation*}

It is well known that the complete elliptic integrals and Gaussian
hypergeometric function have important applications in
quasiconformal mappings, number theory, and other fields of the
mathematical and mathematical physics. For instance, the Gaussian
arithmetic-geometric mean \emph{AGM} and the modulus of the plane
Gr\"{o}tzsch ring can be expressed in terms of the complete elliptic
integrals of the first kind, and the complete elliptic integrals of
the second kind gives the formula of the perimeter of an ellipse.
Moreover, Ramanujan modular equation and continued fraction in
number theory are both related to the Gaussian hypergeometric
function $F(a,b;c;x)$. For these, and the properties for the
complete elliptic integrals and Gaussian hypergeometric function see
[2, 3, 5, 6, 9, 11, 16, 19-21].

For $r\in(0,1)$ and $a\in(0,1)$, the generalized elliptic integrals (see [4, 10])
are defined by
\begin{equation}
{\mathcal{K}}_{a}={\mathcal{K}}_{a}(r)=\frac{\pi}{2}F(a,1-a;1;r^2)
\end{equation}
and
\begin{equation}
{\mathcal{E}}_{a}={\mathcal{E}}_{a}(r)=\frac{\pi}{2}F(a-1,1-a;1;r^2).
\end{equation}
Clearly ${\mathcal{K}}_{a}(0)={\mathcal{E}}_{a}(0)=\pi/2$,
${\mathcal{K}}_{a}(1^-)=\infty$ and
${\mathcal{E}}_{a}(1)=[\sin({\pi}{a})]/[2(1-a)]$. In the particular
case $a=1/2$, the generalized elliptic integrals reduce to the
complete elliptic integrals. By symmetry of (1.1), we assume that
$a\in(0,1/2]$ in the sequence.

The generalized elliptic integrals satisfy the following derivative
formulas:
\begin{equation}
\frac{d\mathcal{K}_{a}}{dr}=\frac{2(1-a)}{r{r'}^2}(\mathcal{E}_{a}-{r'}^2\mathcal{K}_{a}),
\quad
\frac{d\mathcal{E}_{a}}{dr}=-\frac{2(1-a)}{r}(\mathcal{K}_{a}-\mathcal{E}_{a}),
\end{equation}
\begin{equation}
\frac{d}{dr}(\mathcal{K}_{a}-\mathcal{E}_{a})=\frac{2(1-a)r\mathcal{E}_{a}}{{r'}^2},
\quad\frac{d}{dr}(\mathcal{E}_{a}-{r'}^2\mathcal{K}_{a})=2ar\mathcal{K}_{a}.
\end{equation}
Here and in what follows we set  $r^{\prime}={\sqrt {1-r^2}}$ for
$r\in(0,1)$.

In 2000, Anderson, Qiu, Vamanamurthy and Vuorinen [4] reintroduced
the generalized elliptic integrals in geometry function theory,
found that the generalized elliptic integrals of the first kind
$\mathcal{K}_{a}$ arises from the Schwarz-Christoffel transformation
of the upper half-plane onto a parallelogram, and established
several monotonicity theorems for the generalized elliptic integrals
$\mathcal{K}_{a}$ and $\mathcal{E}_{a}$.

Recently, the generalized elliptic integrals have attracted the attention of many mathematicians. In particular,
many remarkable properties and inequalities for the generalized elliptic integrals can be found in the literature [8, 12-14, 22].

Very recently, Takeuchi [18] discussed the generalized trigonometric function and found a new form of the generalized
complete elliptic integrals.

In [16], Qiu and Vamanamurthy proved that the inequality
\begin{equation}
\frac{\mathcal{K}(r)}{\log(4/r')}<1+\frac{1}{4}r'^2
\end{equation}
holds for all $r\in(0,1)$.

Alzer [3] proved that the inequality
\begin{equation}
1+\left(\frac{\pi}{4\log2}-1\right)r'^2<\frac{\mathcal{K}(r)}{\log(4/r')}
\end{equation}
holds for all $r\in(0,1)$. Moreover, Alzer also proved that the constant
factors $1/4$ in (1.5) and $\pi/(4\log2)-1$ in (1.6) are best
possible.

The main purpose of this paper is to generalize inequalities (1.5)
and (1.6) to $\mathcal{K}_{a}$. Our main result is the following
Theorem 1.1.
\begin{thm}
Let $R(x)$ be the Ramanujan constant function defined by (2.1). Then the double inequality
\begin{equation}
1+\alpha r'^2<\frac{\mathcal{K}_{a}(r)}{\sin(\pi
a)\log(e^{R(a)/2}/r')}<1+\beta r'^2
\end{equation}
holds for all $a\in (0,1/2]$ and $r\in(0,1)$ if and only if $\alpha\leq \alpha_{0}=\pi/[R(a)\sin(\pi a)]-1$
and $\beta\geq \beta_{0}=a(1-a)$.
\end{thm}

\bigskip
\section{Some properties for Ramanujan constant function $R(x)$}
\setcounter{equation}{0}

For $x\in(0,1/2]$, the Ramanujan constant function $R(x)$ ([17]) is given
by
\begin{equation}
R(x)=-2\gamma-\Psi(x)-\Psi(1-x),\quad
R\left(1/2\right)=4\log2,
\end{equation}
where
$\gamma=\lim\limits_{n\rightarrow\infty}(\sum_{k=1}^{n}1/k-\log{n})=0.577215\cdots$
is the Euler-Mascheroni constant, and
$$\Psi(x)=\frac{\Gamma'(x)}{\Gamma(x)},\quad \Gamma(x)=\int\limits_{0}^{\infty}t^{x-1}{e}^{-t}dt, \quad x>0.$$

It is well known that $\Psi^{(n)}(x)~~(n\geq 0)$ has the series expansion as follows:
\begin{equation}
\Psi^{(n)}(x)=\left\{
\begin{array}{ll}
\displaystyle-\gamma-\frac{1}{x}+\sum\limits_{k=1}^{\infty}\frac{x}{k(k+x)}, &n=0\\
\displaystyle(-1)^{n+1}n!\sum\limits_{k=0}^{\infty}\frac{1}{(x+k)^{n+1}},
&n\geq 1.
\end{array}
\right.
\end{equation}

The purpose of this section is to present the series expansion formula for
$R(x)(x\in(0,1/2])$ (Theorem 2.2) and establish two important inequalities involving $R(x)$
(Corollaries 2.4 and 2.5), which will be used in the
proof of our main result.

\begin{lem} The function
$\xi(x)=1/[x(1-x)]-R(x)$ is strictly increasing from $(0,1/2]$ onto
$(1,4-4\log2]$.
\end{lem}

\begin{proof} Differentiating $\xi$ yields
\begin{align}
\xi'(x)=&-\frac{1}{x^2}+\frac{1}{(1-x)^2}+\Psi'(x)-\Psi'(1-x)\notag\\
=&-\frac{1}{x^2}+\frac{1}{(1-x)^2}+\sum_{k=0}^{\infty}\frac{1}{(k+x)^2}-\sum_{k=0}^{\infty}\frac{1}{(k+1-x)^2}\notag\\
=&\sum_{k=1}^{\infty}\frac{1}{(k+x)^2}-\sum_{k=1}^{\infty}\frac{1}{(k+1-x)^2}>0\notag
\end{align}
for $x\in(0,1/2]$. Moreover, $\xi(1/2)=4-R(1/2)=4-4\log{2}$ and
\begin{align*}
\lim_{x\rightarrow 0^-}\xi(x)=&\lim_{x\rightarrow
0^-}\left(\frac{1}{x}+\frac{1}{1-x}+\Psi(x)+\Psi(1-x)+2\gamma\right)\\
=&\lim_{x\rightarrow
0^-}\bigg(\frac{1}{x}+\frac{1}{1-x}-\gamma-\frac{1}{x}+\sum_{k=1}^{\infty}\frac{x}{k(k+x)}-\gamma-\frac{1}{1-x}\\
&+\sum_{k=1}^{\infty}\frac{1-x}{k(k+1-x)}+2\gamma\bigg)\\
=&\lim_{x\rightarrow
0^-}\left(\sum_{k=1}^{\infty}\frac{x}{k(k+x)}+\sum_{k=1}^{\infty}\frac{1-x}{k(k+1-x)}\right)=\sum_{k=1}^{\infty}\frac{1}{k(k+1)}=1.
\end{align*}
\end{proof}

\begin{thm} The Ramanujan constant function $R(x)$ has the
following series expansion:
\begin{align}
R(x)=&\frac{1}{x}+2\zeta(3)x^2+2\zeta(5)x^4+\cdots+2\zeta(2k+1)x^{2k}+\cdots\notag\\
=&\frac{1}{x}+\sum\limits_{k=1}^{\infty}2\zeta(2k+1)x^{2k},
\end{align}
where $\zeta(x)=\sum_{n=1}^{\infty}n^{-x}$ is the well-known Riemann zeta function.
\end{thm}

\begin{proof}
Let $f(x)=xR(x)$. Then simple computations lead to
\begin{equation}
\lim_{x\rightarrow 0}f(x)=\lim_{x\rightarrow
0}x\left(\frac{1}{x}+\frac{1}{1-x}-\sum_{k=1}^{\infty}\frac{x}{k(k+x)}-\sum_{k=1}^{\infty}\frac{1-x}{k(k+1-x)}\right)=1,
\end{equation}
\begin{align}
f'(x)=&R(x)+xR'(x)=-\Psi(x)-\Psi(1-x)-2\gamma+x\left[-\Psi'(x)+\Psi'(1-x)\right]\notag\\
=&\frac{1}{x}+\frac{1}{1-x}-\sum_{k=1}^{\infty}\frac{x}{k(k+x)}-\sum_{k=1}^{\infty}\frac{1-x}{k(k+1-x)}\notag\\
&+x\left[-\sum_{k=0}^{\infty}\frac{1}{(k+x)^2}+\sum_{k=0}^{\infty}\frac{1}{(k+1-x)^2}\right]\notag\\
=&\frac{1}{1-x}-\sum_{k=1}^{\infty}\frac{x^2}{k(k+x)^2}-\sum_{k=1}^{\infty}\frac{1-x}{k(k+1-x)}
+\sum_{k=0}^{\infty}\frac{x}{(k+1-x)^2},\notag
\end{align}
\begin{equation}
\lim_{x\rightarrow 0}f'(x)=0.
\end{equation}

For $n\geq 2$, it follows from (2.2) that
\begin{align}
R^{(n)}(x)=&-\Psi^{(n)}(x)-(-1)^n\Psi^{(n)}(1-x)\notag\\
=&-n!(-1)^{n+1}\sum_{k=0}^{\infty}\frac{1}{(x+k)^{n+1}}-(-1)^{n}n!(-1)^{n+1}\sum_{k=0}^{\infty}\frac{1}{(1-x+k)^{n+1}}\notag\\
=&n!\left[\sum_{k=0}^{\infty}\frac{1}{(1-x+k)^{n+1}}+(-1)^n\frac{1}{(x+k)^{n+1}}\right].\notag
\end{align}
Therefore, we get
\begin{align}
f^{(n)}(x)=&nR^{(n-1)}(x)+xR^{(n)}(x)\notag\\
=&n!\sum_{k=0}^{\infty}\left[\frac{1}{(1-x+k)^{n}}+(-1)^{n-1}\frac{1}{(x+k)^{n}}\right]\notag\\
&+n!\sum_{k=0}^{\infty}\left[\frac{x}{(1-x+k)^{n+1}}+(-1)^n\frac{x}{(x+k)^{n+1}}\right]\notag\\
=&n!\sum_{k=0}^{\infty}\left[\frac{1+k}{(1-x+k)^{n+1}}+\frac{(-1)^{n-1}k}{(x+k)^{n+1}}\right].\notag
\end{align}
Furthermore, if $n$ is even, then
\begin{align}
f^{(n)}(x)=&n!\sum_{k=0}^{\infty}\left[\frac{1+k}{(1-x+k)^{n+1}}-\frac{k}{(x+k)^{n+1}}\right]\notag\\
=&n!\sum_{k=0}^{\infty}\left[\frac{1+k}{(1-x+k)^{n+1}}-\frac{1+k}{(x+1+k)^{n+1}}\right],\notag
\end{align}
\begin{equation}
\lim_{x\rightarrow 0}f^{(n)}(x)=0.
\end{equation}
If $n$ is odd, then
\begin{align}
f^{(n)}(x)=&n!\sum_{k=0}^{\infty}\left[\frac{1+k}{(1-x+k)^{n+1}}+\frac{k}{(x+k)^{n+1}}\right]\notag\\
=&n!\sum_{k=0}^{\infty}\left[\frac{1+k}{(1-x+k)^{n+1}}+\frac{1+k}{(x+1+k)^{n+1}}\right],\notag
\end{align}
\begin{equation}
\lim_{x\rightarrow
0}f^{(n)}(x)=2n!\sum_{k=0}^{\infty}\frac{1}{(k+1)^n}=2n!\zeta(n).
\end{equation}

Equations (2.4)-(2.7) implies that $f(x)$ has the following Taylor
series expansion
\begin{equation*}
f(x)=1+2\zeta(3)x^3+2\zeta(5)x^5+\cdots=1+\sum_{k=1}^{\infty}2\zeta(2k+1)x^{2k+1},\
x\in(0,1).
\end{equation*}
Therefore, (2.3) follows.
\end{proof}

\begin{thm} The function
\begin{equation*}
\eta(x)=\frac{\pi/\sin(\pi x)-R(x)}{x(1-x)}
\end{equation*}
is strictly decreasing from $(0,1/2]$ onto $[4\pi-16\log2,\pi^2/6)$.
\end{thm}
\begin{proof} Clearly (2.1) gives $\eta(1/2)=4\pi-16\log2$.
Simple computations lead to
\begin{equation*}
x^2(1-x)^2\eta'(x)=-\left[\frac{\pi^2\cos(\pi x)}{\sin^2(\pi
x)}+R'(x)\right]x(1-x)-\left[\frac{\pi}{\sin(\pi
x)}-R(x)\right](1-2x),
\end{equation*}
\begin{equation}
\eta'\left(\frac{1}{2}\right)=0.
\end{equation}

Making use of the well-known formulas (see [1,
p.75, 4.3.68] and [7, p.16 and p.56])
\begin{equation*}
\frac{1}{\sin{x}}=\frac{1}{x}+\sum_{k=1}^{\infty}\frac{2^{2k}-2}{(2k)!}|B_{2k}|x^{2k-1},\
|x|<\pi,
\end{equation*}
\begin{equation*}
B_{2k}=2(-1)^{k+1}\frac{\zeta(2k)(2k)!}{(2\pi)^{2k}},
\end{equation*}
where $B_{k}$ is the Bernouli number,  we can rewrite $\eta(x)$ as
\begin{align}
\eta(x)=&\frac{\pi x/\sin(\pi x)-xR(x)}{x^2(1-x)}
=\frac{\displaystyle\sum\limits_{k=1}^{\infty}\frac{2^{2k}-2}{2^{2k-1}}\zeta(2k)x^{2k}-\sum\limits_{k=1}^{\infty}2\zeta(2k+1)x^{2k+1}}{x^2(1-x)}\notag\\
=&\frac{\displaystyle\sum\limits_{k=0}^{\infty}\frac{2^{2k+2}-2}{2^{2k+1}}\zeta(2k+2)x^{2k}-\sum\limits_{k=0}^{\infty}2\zeta(2k+3)x^{2k+1}}{(1-x)}.\notag
\end{align}

Therefore, $\eta(0^+)=\zeta(2)=\pi^2/6$ and
\begin{align}
&(1-x)^2\eta'(x)\notag\\
=&\left[\sum\limits_{k=1}^{\infty}\frac{2^{2k+2}-2}{2^{2k+1}}(2k)\zeta(2k+2)x^{2k-1}-\sum\limits_{k=0}^{\infty}2(2k+1)\zeta(2k+3)x^{2k}\right](1-x)\notag\\
&+\sum\limits_{k=0}^{\infty}\frac{2^{2k+2}-2}{2^{2k+1}}\zeta(2k+2)x^{2k}-\sum\limits_{k=0}^{\infty}2\zeta(2k+3)x^{2k+1}\notag\\
=&\sum\limits_{k=0}^{\infty}\frac{2^{2k+4}-2}{2^{2k+3}}(2k+2)\zeta(2k+4)x^{2k+1}-\sum\limits_{k=0}^{\infty}2(2k+1)\zeta(2k+3)x^{2k}\notag\\
&-\sum\limits_{k=0}^{\infty}\frac{2^{2k+4}-2}{2^{2k+3}}(2k+2)\zeta(2k+4)x^{2k+2}+\sum\limits_{k=0}^{\infty}2(2k+1)\zeta(2k+3)x^{2k+1}\notag\\
&+\sum\limits_{k=0}^{\infty}\frac{2^{2k+2}-2}{2^{2k+1}}\zeta(2k+2)x^{2k}-\sum\limits_{k=0}^{\infty}2\zeta(2k+3)x^{2k+1}\notag\\
=&\sum\limits_{k=0}^{\infty}\frac{2^{2k+2}-2}{2^{2k+1}}\zeta(2k+2)x^{2k}+\sum\limits_{k=0}^{\infty}2(2k+1)\zeta(2k+3)x^{2k+1}\notag\\
&+\sum\limits_{k=0}^{\infty}\frac{2^{2k+4}-2}{2^{2k+3}}(2k+2)\zeta(2k+4)x^{2k+1}-\sum\limits_{k=0}^{\infty}2(2k+1)\zeta(2k+3)x^{2k}\notag\\
&-\sum\limits_{k=0}^{\infty}2\zeta(2k+3)x^{2k+1}-\sum\limits_{k=0}^{\infty}\frac{2^{2k+4}-2}{2^{2k+3}}(2k+2)\zeta(2k+4)x^{2k+2}\notag\\
=&:g(x).
\end{align}
Differentiating $g$ yields
\begin{align}
g'(x)=&\sum\limits_{k=1}^{\infty}\frac{2^{2k+2}-2}{2^{2k+1}}(2k)\zeta(2k+2)x^{2k-1}+\sum\limits_{k=0}^{\infty}2(2k+1)^2\zeta(2k+3)x^{2k}\notag\\
&+\sum\limits_{k=0}^{\infty}\frac{2^{2k+4}-2}{2^{2k+3}}(2k+1)(2k+2)\zeta(2k+4)x^{2k}\notag\\
&-\sum\limits_{k=1}^{\infty}2(2k)(2k+1)\zeta(2k+3)x^{2k-1}-\sum\limits_{k=0}^{\infty}2(2k+1)\zeta(2k+3)x^{2k}\notag\\
&-\sum\limits_{k=0}^{\infty}\frac{2^{2k+4}-2}{2^{2k+3}}(2k+2)^2\zeta(2k+4)x^{2k+1}\notag\\
=&\sum\limits_{k=0}^{\infty}\frac{2^{2k+4}-2}{2^{2k+3}}(2k+2)\zeta(2k+4)x^{2k+1}-\sum\limits_{k=0}^{\infty}\frac{2^{2k+4}-2}{2^{2k+3}}(2k+2)^2\zeta(2k+4)x^{2k+1}\notag\\
&+\sum\limits_{k=0}^{\infty}2(2k+1)^2\zeta(2k+3)x^{2k}-\sum\limits_{k=0}^{\infty}2(2k+1)\zeta(2k+3)x^{2k}\notag\\
&+\sum\limits_{k=0}^{\infty}\frac{2^{2k+4}-2}{2^{2k+3}}(2k+1)(2k+2)\zeta(2k+4)x^{2k}\notag\\
&-\sum\limits_{k=0}^{\infty}2(2k+2)(2k+3)\zeta(2k+5)x^{2k+1}\notag\\
=&-\sum\limits_{k=0}^{\infty}\frac{2^{2k+4}-2}{2^{2k+3}}(2k+1)(2k+2)\zeta(2k+4)x^{2k+1}\notag\\
&+\sum\limits_{k=0}^{\infty}2(2k+2)(2k+3)\zeta(2k+5)x^{2k+2}\notag\\
&+\sum\limits_{k=0}^{\infty}\frac{2^{2k+4}-2}{2^{2k+3}}(2k+1)(2k+2)\zeta(2k+4)x^{2k}\notag\\
&-\sum\limits_{k=0}^{\infty}2(2k+2)(2k+3)\zeta(2k+5)x^{2k+1}\notag\\
=&(1-x)g_{1}(x),
\end{align}
where
\begin{align}
g_{1}(x)=&\sum\limits_{k=0}^{\infty}\frac{2^{2k+4}-2}{2^{2k+3}}(2k+1)(2k+2)\zeta(2k+4)x^{2k}\notag\\
&-\sum\limits_{k=0}^{\infty}2(2k+2)(2k+3)\zeta(2k+5)x^{2k+1},\
x\in(0,1/2].
\end{align}
Note that
\begin{align}
g_{1}(x)=&\frac{7}{2}\zeta(4)+\frac{93}{4}\zeta(6)x^2+\sum\limits_{k=2}^{\infty}\frac{2^{2k+4}-2}{2^{2k+3}}(2k+1)(2k+2)\zeta(2k+4)x^{2k}\notag\\
&-12\zeta(5)x-40\zeta(7)x^3-\sum\limits_{k=2}^{\infty}2(2k+2)(2k+3)\zeta(2k+5)x^{2k+1}\notag\\
>&3.788-12.444x+23.653x^2-40.334x^3+\frac{127}{64}\sum\limits_{k=2}^{\infty}(2k+1)(2k+2)x^{2k}\notag\\
&-2\zeta(9)\sum\limits_{k=2}^{\infty}(2k+2)(2k+3)x^{2k+1}\notag\\
>&3.788-12.444x+23.653x^2-40.334x^3+1.984\left[\frac{12x^8-34x^6+30x^4}{(1-x^2)^3}\right]\notag\\
&-2.005\left[\frac{20x^9-54x^7+42x^5}{(1-x^2)^3}\right]\notag\\
=&\frac{1}{1000(1-x^2)^3}\big[3788-12444x+12289x^2-3002x^3-75x^4-540x^5\notag\\
&-285x^6-288x^7+155x^8+234x^9\big].
\end{align}

It is not difficult to verify that the polynomial function
$x\rightarrow
3788-12444x+12289x^2-3002x^3-75x^4-540x^5-285x^6-288x^7+155x^8+234x^9$
is strictly decreasing and positive on $(0,1/2]$. Therefore, $g(x)$
is strictly increasing on $(0,1/2]$, $g(x)<g(1/2)=\eta'(1/2)/4=0$
for $x\in(0,1/2]$ and $\eta(x)$ is strictly decreasing on $(0,1/2]$
follows from (2.8) and (2.9) together with (2.10)-(2.12).
\end{proof}

\begin{cor} The inequality
\begin{equation}
\frac{\pi}{R(x)\sin(\pi x)}-1>\frac{\sin(\pi x)-\pi x(1-x)}{\sin(\pi
x)\left[R(x)-1\right]}
\end{equation}
holds for all $x\in(0,1/2]$.
\end{cor}
\begin{proof} We clearly see that inequality (2.13) is equivalent to
\begin{equation*}
\frac{\pi[1+x(1-x)]R(x)-\pi-\sin(\pi x)R(x)^2}{x(1-x)}>0
\end{equation*}
or
\begin{equation*}
R(x)\sin(\pi x)\frac{\pi/\sin(\pi
x)-R(x)}{x(1-x)}-\pi\frac{1-x(1-x)R(x)}{x(1-x)}>0.
\end{equation*}

It follows from Lemma 2.1 and Theorem 2.3 together with the fact
that $x\rightarrow R(x)\sin(\pi x)$ is strictly decreasing from
$(0,1/2]$ onto $[4\log2,\pi)$ (see [15, Theorem 2]) that
\begin{align*}
&R(x)\sin(\pi x)\frac{\pi/\sin(\pi
x)-R(x)}{x(1-x)}-\pi\frac{1-x(1-x)R(x)}{x(1-x)}\\
>&4\log2(4\pi-16\log2)-\pi(4-4\log2)=(20\log{2}-4)\pi-64\log^2{2}=0.236\cdots
\end{align*}
for $x\in(0,1/2]$.
\end{proof}

\begin{cor} The inequality
\begin{equation}
x(1-x)>\frac{\pi-R(x)\sin(\pi x)}{R(x)\sin(\pi x)}
\end{equation}
holds for all $x\in(0,1/2]$.
\end{cor}

\begin{proof} It follows from Theorem 2.3 that
\begin{align}
&R(x)\sin(\pi x)-\frac{\pi-R(x)\sin(\pi x)}{x(1-x)}>\log16-\sin(\pi
x)\left[\frac{\pi/\sin(\pi x)-R(x)}{x(1-x)}\right]\notag\\
&>\log16-\sin(\pi
x)\frac{\pi^2}{6}>\log16-\frac{\pi^2}{6}=1.127\cdots>0
\end{align}
for $x\in(0,1/2]$. Therefore, inequality (2.14) follows easily from (2.15).
\end{proof}

\bigskip
\section{Proof of Theorem 1.1}
\setcounter{equation}{0}

\begin{lem}(see [5, Theorem 1.25])
For $-\infty<a<b<\infty$, let $f,g:[a,b]\rightarrow{\mathbb{R}}$ be
continuous on $[a,b]$, and be differentiable on $(a,b)$, let
$g'(x)\neq 0$ on $(a,b)$. If $f^{\prime}(x)/g^{\prime}(x)$ is
increasing (decreasing) on $(a,b)$, then so are
$$\frac{f(x)-f(a)}{g(x)-g(a)}\ \ \mbox{and}\ \ \frac{f(x)-f(b)}{g(x)-g(b)}.$$
If $f^{\prime}(x)/g^{\prime}(x)$ is strictly monotone, then the
monotonicity in the conclusion is also strict.
\end{lem}

\begin{lem} The inequality
\begin{equation}
\sin(\pi x)>\frac{\pi x(1-x)}{2}\left[2+x(1-x)\right]
\end{equation}
holds for all $x\in(0,1/2]$.
\end{lem}
\begin{proof} Let
\begin{equation}
h(x)=\sin(\pi x)-\frac{\pi x(1-x)}{2}\left[2+x(1-x)\right].
\end{equation}
Then simple computations lead to
\begin{equation}
h(0^+)=0,
\end{equation}
\begin{equation*}
h'(x)=\pi\left[\cos(\pi x)-(1-2x)(1+x-x^2)\right],
\end{equation*}
\begin{equation}
h'(0^+)=0, \quad h'\left(\frac{1}{2}\right)=0,
\end{equation}
\begin{equation*}
h''(x)=\pi\left[-\pi\sin(\pi x)+1+6x-6x^2)\right],
\end{equation*}
\begin{equation}
h''(0^+)=\pi, \quad
h''\left(\frac{1}{2}\right)=\pi\left(\frac{5}{2}-\pi\right)<0,
\end{equation}
\begin{equation*}
h'''(x)=\pi\left[6-12x-\pi^2\cos(\pi x)\right],
\end{equation*}
\begin{equation}
h'''(0^+)=6-\pi^2<0, \quad h'''\left(\frac{1}{2}\right)=0,
\end{equation}
\begin{equation}
h^{(4)}(x)=\pi\left[-12+\pi^3\sin(\pi x)\right],
\end{equation}
\begin{equation}
h^{(4)}(0^+)=-12\pi, \quad
h^{(4)}\left(\frac{1}{2}\right)=\pi\left(\pi^3-12\right)>0.
\end{equation}

From (3.7) and (3.8) we clearly see that there exists
$x_{0}\in(0,1/2]$ such that $h^{(4)}(x)<0$ for $x\in(0,x_{0})$ and
$h^{(4)}(x)>0$ for $x\in(x_{0},1/2]$. Thus $h'''(x)$ is strictly
decreasing on $(0,x_{0}]$ and strictly increasing on $[x_{0},1/2]$.

Equation (3.6) and the piecewise monotonicity of $h'''(x)$ implies
that $h'''(x)<0$ for $x\in(0,1/2]$. Hence $h''(x)$ is strictly
decreasing on $(0,1/2]$. Then (3.5) leads to the conclusion that there
exists $x_{1}\in(0,1/2]$ such that $h'(x)$ is strictly increasing on
$(0,x_{1}]$ and strictly decreasing on $[x_{1},1/2]$.

It follows from (3.3) and (3.4) together with the piecewise monotonicity of $h'(x)$ that
$h(x)$ is strictly increasing on $(0,1/2]$ and $h(x)>h(0^+)=0$ for $x\in (0,1/2]$.

Therefore, inequality (3.1) follows easily from (3.2) and $h(x)>0$ for $x\in (0,1/2]$.
\end{proof}

\begin{lem} For $a\in(0,1/2]$, defined the function $F$ on $(0,1)$
by
\begin{equation*}
F(r)=\frac{r^{4}\sin(\pi
a)-2a(1-a)r^{2}r'^{2}{\mathcal{K}_{a}}+2(1-a)({r'}^{2}-r^{2})({\mathcal{E}_{a}}-{r'}^{2}{\mathcal{K}_{a}})}{{r'}^{2}r^{4}}.
\end{equation*}
Then $F(r)$ is strictly increasing from $(0,1)$ onto $(\sin(\pi
a)-\pi a(1-a)-(\pi/2)a^{2}(1-a)^{2},a(1-a)\sin(\pi a))$.
\end{lem}

\noindent{\em Proof}. Let $F_{1}(r)=\sin(\pi
a)-2a(1-a)r'^{2}{\mathcal{K}_{a}}/r^2+2(1-a)(r'^{2}/r^2-1)({\mathcal{E}_{a}}-r'^{2}{\mathcal{K}_{a}})/r^2$,
$F_{2}(r)=r'^{2}$. Then $F(r)=F_{1}(r)/F_{2}(r)$,
$F_{1}(1)=F_{2}(1)=0$,
\begin{align*}
F_{1}'(r)=&-2a(1-a)\left[-\frac{2}{r^3}\mathcal{K}_{a}+\frac{r'^2}{r^2}\frac{2(1-a)}{r{r'}^2}(\mathcal{E}_{a}-{r'}^2\mathcal{K}_{a})\right]\\
&+2(1-a)\left(-\frac{2}{r^3}\right)\frac{{\mathcal{E}_{a}}-r'^{2}{\mathcal{K}_{a}}}{r^2}+2(1-a)\left(\frac{r'^2}{r^2}-1\right)
\frac{2ar^2{\mathcal{K}_{a}}-2({\mathcal{E}_{a}}-r'^{2}{\mathcal{K}_{a}})}{r^3}\\
=&4a(1-a)\frac{{\mathcal{K}_{a}}}{r^3}-4a(1-a)^2\frac{{\mathcal{E}_{a}}-{r'}^{2}{\mathcal{K}_{a}}}{r^3}-4(1-a)\frac{{\mathcal{E}_{a}}-{r'}^{2}{\mathcal{K}_{a}}}{r^5}\\
&+4(1-a)(r'^2-r^2)\frac{ar^2{\mathcal{K}_{a}}-({\mathcal{E}_{a}}-r'^{2}{\mathcal{K}_{a}})}{r^5},
\end{align*}
\begin{equation*}
F_{2}'(r)=-2r,
\end{equation*}
\begin{align}
\frac{F_{1}'(r)}{F_{2}'(r)}=&2(1-a)\bigg\{\frac{-ar^2{\mathcal{K}_{a}}+\left[a(1-a)r^2+1\right](\mathcal{E}_{a}-{r'}^2\mathcal{K}_{a})}{r^6}\notag\\
&-\frac{(r'^2-r^2)\left[ar^2{\mathcal{K}_{a}}-({\mathcal{E}_{a}}-r'^{2}{\mathcal{K}_{a}})\right]}{r^6}\bigg\}\notag\\
=&2(1-a)\frac{F_{3}(r)}{r^4},
\end{align}
where
\begin{equation*}
F_{3}(r)=\left[2-2r^{2}+a(1-a)r^{2}\right]\frac{{\mathcal{E}_{a}}-r'^{2}{\mathcal{K}_{a}}}{r^{2}}-2ar'^{2}{\mathcal{K}_{a}}.
\end{equation*}

Making use of series expansion, we get
\begin{align}
\frac{2}{\pi}F_{3}(r)=&\left\{2+[a(1-a)-2]r^2\right\}\sum\limits_{n=0}^{\infty}\frac{a(a)_{n}(1-a)_{n}}{(n+1)!n!}r^{2n}-2a(1-r^2)
\sum\limits_{n=0}^{\infty}\frac{(a)_{n}(1-a)_{n}}{(n!)^2}r^{2n}\notag\\
=&a\bigg\{2\sum\limits_{n=0}^{\infty}\frac{(a)_{n}(1-a)_{n}}{(n+1)!n!}r^{2n}+[a(1-a)-2]\sum\limits_{n=0}^{\infty}\frac{(a)_{n}(1-a)_{n}}{(n+1)!n!}r^{2n+2}\notag\\
&-2\sum\limits_{n=0}^{\infty}\frac{(a)_{n}(1-a)_{n}}{(n!)^2}r^{2n}+2\sum\limits_{n=0}^{\infty}\frac{(a)_{n}(1-a)_{n}}{(n!)^2}r^{2n+2}\bigg\}\notag\\
=&a\bigg\{2\sum\limits_{n=0}^{\infty}\frac{(a)_{n+1}(1-a)_{n+1}}{(n+2)!(n+1)!}r^{2n+2}+[a(1-a)-2]\sum\limits_{n=0}^{\infty}\frac{(a)_{n}(1-a)_{n}}{(n+1)!n!}r^{2n+2}\notag\\
&-2\sum\limits_{n=0}^{\infty}\frac{(a)_{n+1}(1-a)_{n+1}}{[(n+1)!]^2}r^{2n+2}+2\sum\limits_{n=0}^{\infty}\frac{(a)_{n}(1-a)_{n}}{(n!)^2}r^{2n+2}\bigg\}\notag\\
=&a\left[2+a(1-a)\right]\sum\limits_{n=0}^{\infty}\frac{(a)_{n+1}(1-a)_{n+1}}{(n+3)!n!}r^{2n+4}.
\end{align}
From (3.9) and (3.10) one has
\begin{equation}
\frac{F_{1}'(r)}{F_{2}'(r)}=\pi
a(1-a)\left[2+a(1-a)\right]\sum\limits_{n=0}^{\infty}\frac{(a)_{n+1}(1-a)_{n+1}}{(n+3)!n!}r^{2n}.
\end{equation}

Therefore, the monotonicity of $F(r)$ follows from  Lemma 3.1 and
(3.11). Moreover, by l'H\^{o}ptial's rule we get
\begin{equation*}
\lim\limits_{r\rightarrow 1^{-}}F(r)=a(1-a)\sin(\pi a),
\end{equation*}
\begin{align*}
\lim\limits_{r\rightarrow 0^{+}}F(r)=&\sin(\pi
a)+\lim\limits_{r\rightarrow
0^{+}}\frac{2(1-a)(-r^{2})({\mathcal{E}_{a}}-r'^{2}{\mathcal{K}_{a}})}{r^{4}}\notag\\
&+\lim\limits_{r\rightarrow \notag
0^{+}}\frac{2(1-a)r'^{2}({\mathcal{E}_{a}}-r'^{2}{\mathcal{K}_{a}})-2a(1-a)r^{2}r'^{2}{\mathcal{K}_{a}}}{r^{4}}\notag\\
=&\sin(\pi a)-2(1-a)\frac{\pi a}{2}+\lim\limits_{r\rightarrow
0^{+}}2(1-a)\frac{({\mathcal{E}_{a}}-r'^{2}{\mathcal{K}_{a}})-ar^{2}{\mathcal{K}_{a}}}{r^{4}}\notag\\
=&\sin(\pi a)-\pi a(1-a)+2(1-a)\lim\limits_{r\rightarrow
0^{+}}\frac{-2(1-a)ar^{2}({\mathcal{E}_{a}}-r'^{2}{\mathcal{K}_{a}})/(rr'^{2})}{4r^{3}}\notag\\
=&\sin(\pi a)-\pi a(1-a)-(\pi/2)a^{2}(1-a)^{2}.\quad \Box
\end{align*}

\begin{lem}
Let $a\in(0,1/2]$, $\lambda_{1}=[\sin(\pi
a)-a(1-a)\pi]/\{\sin(\pi a)[R(a)-1]\}$, $\lambda_{2}=1-[\pi
a(1-a)]/\sin(\pi a)-[\pi a^{2}(1-a)^{2}]/[2\sin(\pi a)]$,
$\beta=a(1-a)$, and the function $G_{\lambda}(\cdot)$ be defined on $(0, 1)$ by
\begin{equation*}
G_{\lambda}(r)=\frac{1}{2}-\log\left(\frac{e^{R(a)/2}}{r'}\right)+\frac{r^{2}\sin(\pi
a)-2(1-a)({\mathcal{E}_{a}}-{r'}^{2}{\mathcal{K}_{a}})} {2\lambda
r^{2}{r'}^{2}\sin(\pi a)}.
\end{equation*}
Then the following statements are true:

(1) $G_{\lambda}'(r)>0$ for all $r\in(0,1)$ if $0<\lambda\leq\lambda_{2}$;

(2) $G_{\lambda}'(r)<0$ for all $r\in(0,1)$ if $\lambda\geq\beta$;

(3) There exists $r_{0}\in(0,1)$ such that $G_{\lambda}'(r)<0$ for $r\in(0,r_{0})$
and $G_{\lambda}'(r)>0$ for $r\in(r_{0},1)$ if $\lambda_{2}<\lambda<\beta$.

Moreover,
\begin{equation}
\left\{
\begin{array}{ll}
G_{\lambda}(0^{+})>0,&\quad 0<\lambda<\lambda_{1},\\
G_{\lambda}(0^{+})=0,&\quad\lambda=\lambda_{1},\\
G_{\lambda}(0^{+})<0,&\quad\lambda>\lambda_{1}
\end{array}
\right.
\end{equation}
and
\begin{equation}
\left\{
\begin{array}{ll}
G_{\lambda}(1^{-})=+\infty,&\quad0<\lambda<\beta,\\
G_{\lambda}(1^{-})=1-\displaystyle\frac{1}{2a(1-a)}<0,&\quad\lambda=\beta,\\
G_{\lambda}(1^{-})=-\infty,&\quad\lambda>\beta.
\end{array}
\right.
\end{equation}
\end{lem}

\begin{proof} It is apparent from Lemma 3.2 that $\lambda_{2}>0$ for all $a\in(0,1/2]$. Therefore, parts (1)-(3) follows from Lemma 3.3 and the fact that
\begin{align*}
&G_{\lambda}'(r)=-\frac{r}{r'^{2}}+\frac{1}{2\lambda\sin(\pi a)}\notag\\
&\times\frac {[2r\sin(\pi a)-4a(1-a)r{\mathcal{K}_{a}}]r^{2}r'^{2}-
[r^{2}\sin(\pi
a)-2(1-a)({\mathcal{E}_{a}}-r'^{2}{\mathcal{K}_{a}})](2rr'^{2}-2r^{3})}
{r^{4}r'^{4}}\notag\\
&=\frac{r}{\lambda r'^{2}}\left[-\lambda+\frac{1}{\sin(\pi
a)}\frac{r^{4}\sin(\pi
a)-2a(1-a)r^{2}r'^{2}{\mathcal{K}_{a}}+2(1-a)(r'^{2}-r^{2})({\mathcal{E}_{a}}-r'^{2}{\mathcal{K}_{a}})}{r'^{2}r^{4}}\right]\notag\\
&=\frac{r}{\lambda r'^{2}}\left[-\lambda+\frac{1}{\sin(\pi
a)}F(r)\right],
\end{align*}
where $F$ is defined as in Lemma 3.3.

Next, we calculate the limit values of $G_{\lambda}(r)$ at $0$ and
$1$. Simple computations lead to
\begin{align}
G_{\lambda}(0^{+})=&\frac{1}{2}-\frac{R(a)}{2}+\frac{\sin(\pi a)-2(1-a)[(\pi a)/2]}{2\lambda \sin(\pi a)}\notag\\
=&\frac{\sin(\pi a)[R(a)-1]}{2\lambda \sin(\pi a)}\left(-\lambda+\frac{\sin(\pi a)-\pi a(1-a)}{ \sin(\pi a)[R(a)-1]}\right)\notag\\
=&\frac{\sin(\pi a)[R(a)-1]}{2\lambda \sin(\pi
a)}(-\lambda+\lambda_{1}),\notag
\end{align}
which implies (3.12).

It follows from
\begin{equation*}
\lim\limits_{r\rightarrow 1^-}\frac{\mathcal{K}_{a}(r)}{\sin(\pi
a)}-\log\left(e^{R(a)/2}/r'\right)=0
\end{equation*}
in [17, p. 635, (2.26)] that
\begin{align}
G_{\lambda}(1^-)=&\frac{1}{2}+\lim\limits_{r\rightarrow
1^{-}}\left[\frac{{\mathcal{K}_{a}}}{\sin(\pi
a)}-\log(e^{R(a)/2}/r')\right]+\lim\limits_{r\rightarrow
1^{-}}\left[\frac{r^2\sin(\pi a)-\sin(\pi a)}{2\sin(\pi a)\lambda
r^{2}r'^{2}}\right]\notag\\
&+\lim\limits_{r\rightarrow 1^{-}}\left[\frac{\sin(\pi
a)-2(1-a)({\mathcal{E}_{a}}-r'^{2}{\mathcal{K}_{a}})}{2\sin(\pi
a)\lambda r^{2}r'^{2}}-\frac{\mathcal{K}_{a}(r)}{\sin(\pi a)}\right]\notag\\
=&\frac{1}{2}-\frac{1}{2\lambda}+\lim\limits_{r\rightarrow
1^{-}}\left[\frac{\sin(\pi
a)-2(1-a)({\mathcal{E}_{a}}-r'^{2}{\mathcal{K}_{a}})}{2\sin(\pi
a)\lambda r^{2}r'^{2}}-\frac{\mathcal{K}_{a}(r)}{\sin(\pi
a)}\right].
\end{align}

We divide the proof of (3.13) into two cases.

{\bf Case 1} $\lambda\neq \beta$. Then from (3.14) we have
\begin{align*}
G_{\lambda}(1^{-})=&\frac{1}{2}-\frac{1}{2\lambda}+\lim\limits_{r\rightarrow
1^{-}}\left[-1+\frac{\sin(\pi
a)-2(1-a)({\mathcal{E}_{a}}-r'^{2}{\mathcal{K}_{a}})}{r'^{2}\mathcal{K}_{a}}\frac{1}{2\lambda
r^{2}}\right]\frac{\mathcal{K}_{a}(r)}{\sin(\pi
a)}\\
=&\frac{1}{2}-\frac{1}{2\lambda}+\lim\limits_{r\rightarrow
1^{-}}\left[-1+\frac{a(1-a)}{\lambda}\right]\frac{\mathcal{K}_{a}(r)}{\sin(\pi
a)}\\
=&\left\{
\begin{array}{ll}
+\infty,&\quad0<\lambda<\beta,\\
-\infty,&\quad\lambda>\beta.
\end{array}
\right.
\end{align*}

{\bf Case 2} $\lambda=\beta$. Then equation (3.14) leads to
\begin{align*}
G_{\beta}(1^{-})=&\frac{1}{2}-\frac{1}{2a(1-a)}
+\lim\limits_{r\rightarrow 1^{-}}\frac{\sin(\pi
a)-2(1-a)({\mathcal{E}_{a}}-r'^{2}{\mathcal{K}_{a}})-2a(1-a)r^{2}r'^{2}{\mathcal{K}_{a}}}{2a(1-a)r^{2}r'^{2}\sin(\pi
a)}\\
=&\frac{1}{2}-\frac{1}{2a(1-a)}+\frac{1}{a\sin(\pi
a)}\lim\limits_{r\rightarrow 1^{-}}\frac{\displaystyle\frac{\sin(\pi
a)}{2(1-a)}-({\mathcal{E}_{a}}-r'^{2}{\mathcal{K}_{a}})-ar^{2}r'^{2}{\mathcal{K}_{a}}}{r'^{2}}\\
=&\frac{1}{2}-\frac{1}{2a(1-a)}\notag\\
&+\frac{1}{a\sin(\pi a)}\lim\limits_{r\rightarrow
1^{-}}\frac{-2ar{\mathcal{K}_{a}}-2arr'^{2}{\mathcal{K}_{a}}+2ar^{3}{\mathcal{K}_{a}}-2a(1-a)r
({\mathcal{E}_{a}}-r'^{2}{\mathcal{K}_{a}})}{-2r}\\
=&\frac{1}{2}-\frac{1}{2a(1-a)}+\frac{1}{\sin(\pi
a)}\lim\limits_{r\rightarrow
1^{-}}\left[2r'^{2}{\mathcal{K}_{a}}+(1-a)({\mathcal{E}_{a}}-r'^{2}{\mathcal{K}_{a}})\right]\\
=&1-\frac{1}{2a(1-a)}.
\end{align*}
\end{proof}

\begin{proof}[\sc \textbf{Proof of Theorem 1.1}]
Let
\begin{equation}
H_{\lambda}(r)=\sin(\pi a)(1+\lambda
r'^{2})\log(e^{R(a)/2}/r')-{\mathcal{K}_{a}(r)},\quad\lambda \in
\mathbb{R}^{+}.
\end{equation}
Then simple computations lead to
\begin{equation}
H_{\lambda}(0^{+})=[\sin(\pi a)(1+\lambda)-\pi]/2,
\end{equation}
\begin{equation}
H_{\lambda}(1^{-})=0,
\end{equation}
\begin{align}
H_{\lambda}'(r)=&\sin(\pi a)(-2\lambda
r)\log(e^{R(a)/2}/r')+\sin(\pi a)(1+\lambda
r'^{2})\left(\frac{r}{r'^{2}}\right)\notag\\
&-2(1-a)
\frac{({\mathcal{E}_{a}}-r'^{2}{\mathcal{K}_{a}})}{rr'^{2}} \notag\\
=&2\lambda r\sin(\pi a)G_{\lambda}(r),
\end{align}
where $G_{\lambda}$ is defined as in Lemma 3.4.

We divide the proof of inequality (1.7) into two cases.

{\bf Case I} $\lambda=\alpha_{0}=\pi/[R(a)\sin(\pi a)]-1$. Then equation (3.16) reduces to
\begin{equation}
H_{\alpha_{0}}(0^{+})=0.
\end{equation}

From Corollaries 2.4 and 2.5 we know that
$\beta>\alpha_{0}>\lambda_{1}$, then (3.12) and (3.13) lead to the
conclusion that $G_{\alpha_{0}}(0^+)<0$ and
$G_{\alpha_{0}}(1^-)=+\infty$. Moreover, wether
 $\alpha_{0}\in(0,\lambda_{2}]$ or
$\alpha_{0}\in(\lambda_{2},\beta)$, it follows from part (1) or (3)
in Lemma 3.4 that there exists $r_{0}^*\in(0,1)$ such that
$G_{\alpha_{0}}(r)<0$ for $r\in(0,r_{0}^*)$ and
$G_{\alpha_{0}}(r)>0$ for $r\in(r_{0}^*,1)$. Hence, from (3.18) we
clearly see that $H_{\alpha_{0}}(r)$ is strictly decreasing on
$(0,r_{0}^*)$ and strictly increasing on $(r_{0}^*,1)$.

Equations (3.17) and (3.19) together with the piecewise monotonicity
of $H_{\alpha_{0}}(r)$ lead to the conclusion that $H_{\alpha_{0}}(r)<0$ for all $r\in(0,1)$. Therefore,
the first inequality in (1.7) for $\alpha=\alpha_{0}$ follows easily from (3.15).

{\bf Case II} $\lambda=\beta_{0}=a(1-a)$. Then from (3.12), (3.13), Lemma
3.4(2) and the fact that $\beta_{0}>\lambda_{1}$ for all $a\in(0,1/2]$
we know that $G_{\beta_{0}}(r)$ is strictly decreasing on $(0,1)$,
$G_{\beta_{0}}(0^+)<0$ and $G_{\beta_{0}}(1^-)<0$. Thus $G_{\beta_{0}}(r)<0$ for
$r\in(0,1)$. It follows from (3.17) and (3.18) that $H_{\beta_{0}}(r)$
is strictly decreasing on $(0,1)$ and $H_{\beta_{0}}(r)>H_{\beta_{0}}(1^-)=0$ for $r\in(0,1)$. Therefore, the second
inequality in (1.7) for $\beta=\beta_{0}$ follows from (3.15).

\medskip
Finally, we prove that $\alpha=\alpha_{0}$ and $\beta=\beta_{0}$ are the best possible
parameters such that inequality (1.7) holds for all $a\in (0, 1/2]$ and $r\in(0,1)$. In
fact, if $\lambda>\alpha_{0}$, then from (3.16) we know that $H_{\lambda}(0^+)>0$.
Hence there exists $r_{1}\in(0,1)$ such that $H_{\lambda}(r)>0$ for
$r\in(0,r_{1})$. That is, $1+\lambda
r'^2>\mathcal{K}_{a}(r)/[\sin(\pi a)\log(e^{R(a)/2}/r')]$ for
$r\in(0,r_{1})$.

On the other hand, if $0<\lambda<\beta_{0}$, then (3.13) and (3.18)
imply that there exists $r_{1}^*\in (0,1)$ such that
$H_{\lambda}'(r)>0$ for $r\in(r_{1}^*,1)$. Thus $H_{\lambda}(r)$ is
strictly increasing on $(r_{1}^*,1)$ and
$H_{\lambda}(r)<H_{\lambda}(1^-)=0$. That is, $1+\lambda
r'^2<\mathcal{K}_{a}(r)/[\sin(\pi a)\log(e^{R(a)/2}/r')]$ for
$r\in(r_{1}^*,1)$.
\end{proof}

\end{document}